\documentclass[11pt]{article}
\usepackage{amsmath,amsthm,amssymb,color,verbatim,graphicx,framed}
\newcommand{\remove}[1]{}
\newtheorem{thm}{Theorem}

\newtheorem{lem}[thm]{Lemma}

\newtheorem{cor}[thm]{Corollary}

\newtheorem{conjecture}{Conjecture}

\newtheorem{remark}{Remark}

\renewcommand{\remove}[1]{}

\newcommand{\poly}{{\rm poly}}

\newcommand{\eps}{{\varepsilon}}
\renewcommand{\l}{\left}
\renewcommand{\r}{\right}

\newcommand{\comments}[1]{}

\def\F{{\mathbb{F}}}

\newcommand{\Z}{\mathbb{Z}}

\newcommand{\R}{\mathbb{R}}

\renewcommand{\F}{\mathbb{F}}

\newcommand{\Fq}{\mathbb{F}_q}
\newcommand{\calP}{\mathcal{P}}
\newcommand{\calL}{\mathcal{L}}
\newcommand{\calM}{\mathcal{M}}
\newcommand{\calN}{\mathcal{N}}
\newcommand{\calQ}{\mathcal{Q}}
\newcommand{\calS}{\mathcal{S}}

\def\draft{0}   

\ifnum\draft=1 
    \def\ShowAuthNotes{1}
\else
    \def\ShowAuthNotes{0}
\fi

\ifnum\ShowAuthNotes=0
\newcommand{\authnote}[2]{{ \footnotesize \bf{\color{red}[#1's Note: {\color{blue}#2}]}}}
\else
\newcommand{\authnote}[2]{}
\fi

\begin{document}
\title{On primitive elements in finite fields of low characteristic}

\author{
Abhishek Bhowmick\thanks{Email: bhowmick@cs.utexas.edu. Department of Computer Science, The University of Texas at Austin. Research supported in part by NSF Grants CCF-0916160 and CCF-1218723.}
\and
Th\'ai Ho\`ang L\^e\thanks{Email: leth@math.utexas.edu. Department of Mathematics, The University of Texas at Austin.}
}

\maketitle
\thispagestyle{empty}
\begin{abstract}
We discuss the problem of constructing a small subset of a finite field containing primitive elements of the field. Given a finite field, $\F_{q^n}$, small $q$ and large $n$, we show that the set of all low degree polynomials contains the expected number of primitive elements.

The main theorem we prove is a bound for character sums over short intervals in function fields. Our result is unconditional and slightly better than what is known (conditionally under GRH) in the integer case and might be of independent interest.
\end{abstract}

\section{Introduction}
\subsection{Character sum estimates}
Let $q>1$ and $\chi$ be a non-principal character modulo $q$. It is desirable to have an estimate of the form
\[
  \sum_{1 \leq n \leq x} \chi(n) = o(x)
\]
for $x$ as small as possible (depending on $q$). Such an estimate would have immediate applications in finding the first non-quadratic residue modulo $q$, 
or the smallest primitive root modulo $q$ (if, say, $q$ is an odd prime power).

In \cite{mv}, Montgomery and Vaughan had the idea of comparing the above character sum to the corresponding sum over \textit{smooth numbers}. 
Given $x, y>0$, a positive integer $n$ is called $y$-smooth if none of its prime divisors are greater than $y$. Let $\calS(x,y)$ be the set of all integers $1 \leq n \leq x$ that is $y$-smooth. For a function $f$, put
$\Psi(x,y;f) = \sum_{n \in \calS(x,y)} f(n)$. We also put $\Psi(x,y) = \Psi(x,y;1) = |\calS(x,y)|$.

Montgomery and Vaughan \cite{mv} proved the following estimate under the Generalized Riemann Hypothesis (GRH):

\begin{thm}[{\cite[Lemma 2]{mv}}]
 Let $\chi$ be a non-principal character modulo $q$. Then for all $(\log q)^4 \leq y \leq x \leq q$, we have
 \[
  \sum_{1 \leq n \leq x} \chi(n) = \sum_{n \in \calS(x,y)} \chi(n) + O \left(xy^{-1/2} (\log q)^4 \right).
 \]
\end{thm}
Coupled with known estimates on $\Psi(x,y)$, any result of this kind will lead to an estimate for $\sum_{n \leq x} \chi(n)$. 

In \cite{gs}, Granville and Soundararajan refined Montgomery-Vaughan's method and improved their result to a wider range of $x$ and $y$. 
They proved the following, still under GRH:

\begin{thm}[{\cite[Theorem 2]{gs}}]
 Let $\chi$ be a non-principal character modulo $q$. Then for all $(\log q)^2 (\log x)^2 (\log \log x)^{12} \leq y \leq x \leq q$, we have
 \[
  \sum_{1 \leq n \leq x} \chi(n) = \sum_{n \in \calS(x,y)} \chi(n) + O \left( \frac{\Psi(x,y)}{(\log \log q)^2} \right).
 \]
\end{thm}
They also conjectured the following estimate for a wider range.  
\begin{conjecture}
 There exists a constant $A>0$ such that the following holds. Let $\chi$ be a non-principal character modulo $q$. 
 Then for all $(\log q + (\log x)^2 ) (\log \log q)^A \leq y \leq x \leq q$, we have
 \[
  \sum_{1 \leq n \leq x} \chi(n) = \sum_{n \in \calS(x,y)} \chi(n) + o \left( \Psi(x,y, \chi_0) \right).
 \]
 Here $\chi_0$ is the principal character modulo $q$.
\end{conjecture}

In this paper, we will prove an analog of these estimates in function fields. Let $\F_q$ be the field over $q$ elements and $\Fq[t]$ be the polynomial ring over $\Fq$. 
Let $A_d$ denote the set of all polynomials of degree \textit{exactly} $d$. Given $r \in \Z^+$, a polynomial $f \in \Fq[t]$ is called $r$-\textit{smooth} 
if none of its irreducible factors have degree greater than $r$. Let $\calP(d,r)$ denote the set of all $r$-smooth polynomials in $A_d$ and $N(d,r)=|\calP(d,r)|$.

If $Q \in \Fq[t]$, then a character $\chi$ modulo $Q$ is simply a character on the multiplicative group $\left(\Fq[t]/(Q)\right)^{\times}$, 
which can be extended to a function on all of $\Fq[t]$ by periodicity. If $Q$ is irreducible of degree $n$, then $\chi$ can be naturally regarded as a character 
on the field $\F_{q^n}$.

Thoughout this paper, $f$ will stand for a \textit{monic} polynomial and $P$ for a \textit{monic, irreducible} polynomial in $\Fq[t]$.

We will prove the following:

\begin{thm} \label{th:main}
Let $Q$ be a polynomial of degree $n$ in $\Fq[t]$ and $\chi$ be a non-principal character modulo $Q$. 
For any $2 \log_q n \leq r \leq d \leq n$, we have
\[
  \sum_{f \in A_d} \chi(f) = \sum_{f \in \calP(d,r)} \chi(f) + O \left( n q^{-r/2} q^d\right).
\]
\end{thm}
Notice that the implicit constant in $O(\cdot)$ is independent of $q$. The range of applicability of our estimate is better than that of Granville-Soundararajan, which corresponds to
\[
 2 \log_q n + 2 \log_q d + O(\log_q \log_q d) \leq r \leq d \leq n.
\]
Our error term is also better than that of Montgomery-Vaughan, which corresponds to
\[
 O\left( n^4 q^{-r/2} q^d \right).
\]
However, our range is still far weaker than the conjectured range of Granville-Soundararajan.

Our method follows Montgomery-Vaughan closely. The sources of our improvements are a character sum estimate not available in the integers (Theorem \ref{th:hsu}), 
and the use of Cauchy's integral formula instead of Perron's formula. 

From Theorem \ref{th:main} we deduce the following:
\begin{cor}\label{cor:main} Let $Q$ be a polynomial of degree $n$ in $\Fq[t]$ and $\chi$ be a non-principal character modulo $Q$. 
For any $2 \log_q n  \leq r \leq d \leq n$, we have
\[
  \frac{1}{q^d} \left|\sum_{f \in A_d} \chi(f)\right| \leq \eps_{q,d,r,n},
\]
where $\eps_{q,d,r,n}=\rho \left( \frac{d}{r} \right) q^{O \left( \frac{d\log d}{r^2} \right) } + O \left(n q^{-r/2}\right)$ and $\rho(\cdot)$ is the Dickman function. 
\end{cor}
Like Theorem \ref{th:main}, this inequality is uniform in $q$. The function $\rho$ will be defined later in Section \ref{sec:smooth}, but for now we note that $\rho$ decays extremely rapidly to 0: $\rho(u) = u^{-u (1+o(1))}$ as $u \rightarrow \infty$.

\subsection{Finding primitive roots in $\F_{q^n}$}
The problem of deterministically outputting a primitive element of a finite field is a notoriously hard problem. However, a relaxation of this problem is well studied. The goal is to output a small subset of the field guaranteed to contain primitive elements (even one primitive element). This has applications in coding theory, cryptography and combinatorial designs. For details, see the wonderful survey of Shparlinski (Research problem AP5, \cite{Shparsurvey}). We mention the relevant literature now. 

Let $\F_{q^n}$ be the finite field in consideration. We consider the setting of small $q$ and growing $n$. Shoup \cite{Shoup} and Shparlinski \cite{Shpar} independently gave efficient algorithms to construct a set $M \subset \F_{q^n}$ when $q$ is prime\footnote{Shoup's result on character sums also works for $q$ non prime though the result is stated for prime $q$.} of size $\poly(n,q)$ that contains a primitive element. Shparlinski's set is of size $O(n^{10})$. Shoup uses a stronger sieve and gets a set of smaller size (In fact, as noted in \cite{Shparsurvey}, using the stronger sieve, Shparlinksi also obtains a similar bound). However, Shoup outputs the set of low degree polynomials that are irreducible and guarantees a lower bound on the density of the primitive elements, say $\mu$ (where the density is taken over the set of irreducible polynomials). However, the enumeration step requires outputting all degree $d$ polynomials and therefore the density suffers a loss ($o(\mu)$).

Our result differs from this in the following sense. Firstly, we show that the simple set of all monic polynomials of degree $d \sim C\log_q n$ is guaranteed to contain a primitive element. This statement as it is, is not new because of the above result of Shoup. However, we also show that the density of primitive elements is $\mu$, that is we do not suffer any density loss.  We also give another density argument which basically states for $d$ as above the density of primitive elements in the set of monic degree $d$ polynomials is close to $\phi(q^n-1)/(q^n-1)$ (where $\phi(\cdot)$ is the Euler totient function) when $q^n-1$ does not have too many distinct prime divisors. We state our application in more detail next.

From our character sum estimates, we derive the following.
Let $N=q^n$. 
Note that for an irreducible polynomial $Q$ of degree $n$, $\F_{q^n}\equiv \Fq[x]/(Q(x))$. Let $\omega(N-1)$ be the number of distinct prime factors of $N-1$. Let $\calQ(d)$ be the set of monic polynomials of degree $d$ that are primitive in $\F_{q^n}$. We prove the following:
 
\begin{thm}\label{th:main2}Let $d=\l(2\log_q n+2\log_q(1/\eps)\r)\frac{C \log 1/\eps}{\log \log 1/\eps}$, for some absolute constant $C$. Then, $$\l|\frac{|\calQ(d)|}{q^d}-\frac{\phi(N-1)}{N-1}\r|= 2^{\omega(N-1)}O_q(\eps).$$
\end{thm}

Another way of looking at this theorem is to see that we have a probabilistic algorithm to output a primitive element that improves on the naive algorithm of outputting a random element in the following ways and gives similar success probability. 
\begin{itemize}
\item Uses nearly logarithmic (in $n$) randomness as opposed to linear.
\item Runs in logarithmic time as opposed to linear.
\end{itemize}

Note that if $N-1$ is prime, then $\omega(N-1)=1$. However, if $\omega(N-1)$ is large, then our bound becomes trivial due to the exponential dependence on $\omega(N-1)$. That is, we require $\eps\ll2^{-\omega(N-1)}$. In such cases, we prove a stronger lower bound on the density of primitive elements using Iwaniec's shifted sieve (as used in \cite{Shoup}). As one can observe, we need a much weaker dependence on $\eps$ now, that is $\eps\ll\omega(N-1)^{-2}$. However, we no longer have a pseudorandomness statement as we have a one sided guarantee which is needed in most applications.
We state it next.
\begin{thm}\label{thm:iwimprove}There is a universal constant $c$ such that the following is true. Let $\eps=\eps_{q,d,r,n}$ from Corollary~\ref{cor:main}. Let $N=q^n$. Then, if $\eps<\frac{c}{\omega(N-1)^2 (\log \omega(N-1) + 1)^2}$, then $$\frac{|\calQ(d)|}{q^d} \geq \frac{c}{(\log \omega(N-1) + 1)^2}.$$ 
\end{thm}

\section{Preliminaries}
\subsection{$L$-functions in $\Fq[t]$}
Recall that $A_k$ is the set of all monic polynomials of degree exactly $k$ in $\Fq[t]$.  
Let $I_k$ be the subset of $A_k$ consisting of irreducible polynomials and $\pi_k=|I_k|$. It is well known that
\begin{equation} \label{eq:pi1}
 \pi_k \leq \frac{q^k}{k}.
\end{equation}

Following Montgomery-Vaughan, we will work with $L$-functions in $\Fq[t]$. Fix $Q$ of degree $n$ and a non-principal character $\chi$ modulo $Q$. Define
\begin{equation}
 L(s, \chi)=\sum_{f \textup{ monic}} \frac{\chi(f)}{|f|^s}
\end{equation}
for $|s|>1$. Put $A(d, \chi)=\sum_{f \in A_d} \chi(f)$. Then we can write
\[
L(s, \chi)=\sum_{m=0}^{\infty} A(m,\chi)q^{-ms}.
\]
Actually, it is even more convenient to put
\[
 \calL(z, \chi)= \sum_{m=0}^{\infty} A(m,\chi) z^{m}.
\]
Clearly $L(s,\chi)=\calL(q^{-s},\chi)$ for any $\textup{Re}(s)>1$. Since $A(m,\chi)=0$ whenever $m \geq n$ \cite[Proposition 4.3]{rosen}, $\calL(z, \chi)$ is a polynomial of degree at most $n-1$, and in particular an entire function. 
We have the Euler product formula
\begin{equation} \label{eq:euler2}
 \calL(z,\chi)=\prod_{P} \left( 1 - \chi(P) z^{\deg(P)} \right)^{-1}
\end{equation}
whenever $|z|<1/q$. 

The \textit{Generalized Riemann Hypothesis in function fields}, proved by Weil, states that all roots of $\calL$ have modulus equal to either 1 or $q^{-1/2}$.
Thus we have yet another representation
\begin{equation} \label{eq:riemann}
 \calL(z,\chi)=\prod_{i=1}^m (1 - \alpha_i z)
\end{equation}
where $1 \leq m \leq n-1$ and $|\alpha_i|=1$ or $q^{1/2}$ for any $i=1,\ldots, m$.

We recall the following results in $\Fq[t]$ which we will need.

\begin{lem}[Mertens' estimate \cite{rosen2}] \label{mertens} For any $k>0$, we have
\[
\prod_{\deg P \leq k} \left(1-q^{-\deg P} \right)^{-1}=e^\gamma k (1+o_k(1))
\] 
where $\gamma$ is Euler's constant.
\end{lem}

The next result is a character sum estimate which is not available in the integers. 

\begin{thm}[{\cite[Theorem 2.1]{hsu}}] \label{th:hsu}
 For any polynomial $Q$ of degree $n$ in $\Fq[t]$, any non-trivial character $\chi$ modulo $Q$ and $k>0$, we have
 \[
  \left| \sum_{P \in I_k} \chi(P) \right| \leq (n+1) \frac{q^{k/2}}{k}.
 \]
\end{thm}
Hsu \cite{hsu} attributed this result to Effinger and Hayes \cite{eh}. When $\chi$ is quadratic, a more general result was proved by Car \cite[Proposition II.2]{car}. 
For completeness, we include a quick proof of Theorem \ref{th:hsu}.
\begin{proof}
 From (\ref{eq:euler2}) and (\ref{eq:riemann}), by taking logarithmic derivatives, we have
 \[
  \sum_{f \in \Fq[t]} \Lambda(f) \chi(f) z^{\deg f} = \sum_{l=1}^\infty \left( - \sum_{i=1}^m \alpha_i^{l} \right) z^{l}
 \]
for $|z|<1/q$, where $\Lambda(f)$ is the von Mangoldt function
\[
\Lambda(f)=
\left\{
  \begin{array}{ll}
    \deg(P), & \hbox{if $P = P^k$ for some monic, irreducible $P$;} \\
    0, & \hbox{otherwise.}
  \end{array}
\right.
\]
By comparing the coefficients of $z^k$, and using the fact that $|\alpha_i| \leq q^{1/2}$ for each $i$, we have
\[
 \left| \sum_{f \in A_k} \Lambda(f) \chi(f) \right| \leq (n-1)q^{k/2}.
\]
Therefore,
\[
 \left| \sum_{l|k} l \sum_{P \in I_l} \chi(P^{k/l}) \right| \leq (n-1)q^{k/2}.
\]
Consequently, in view of (\ref{eq:pi1}) 
\begin{eqnarray*}
 \left| k \sum_{P \in I_k}  \chi(P) \right| &\leq& (n-1)q^{k/2} + \sum_{l|k, l<k} l \pi_l \\
 &\leq & (n-1)q^{k/2} + \sum_{l|k,l<k} q^l \\
 &\leq & n q^{k/2} + \sum_{l=1}^{[k/2]-1} q^l \\
 &\leq & \left(n+1\right)q^{k/2} \\
\end{eqnarray*}
as desired.
\end{proof}

\subsection{Smooth polynomials} \label{sec:smooth}
We recall some facts about the function $N(r,d)$. Just as its integer counterpart $\Psi(x,y)$, $N(d,r)$ 
is closely related to the \textit{Dickman function}. Recall that the Dickman function $\rho(u)$ is the unique continuous function satisfying 
\[
\rho(u) = \frac{1}{u} \int_{u-1}^u \rho(t) dt
\]
for all $u>1$, with initial condition $\rho(u)=1$ for $0 \leq u \leq 1$. 
We have $\rho(u)=1-\log u$ for $1 \leq u \leq 2$, and the following asymptotic formula \cite[Corollary2.3]{HT}
\[
\rho(u) = \exp \left( -u \left( \log u + \log \log (u+2) - 1 + O \left( \frac{\log \log (u+2)}{\log (u+2)} \right) \right) \right)
\]
for all $u \geq 1$. In particular, we have
\begin{equation} \label{eq:rho}
 \rho(u) \leq \exp \left( -u \log u \right)
\end{equation}
for $u$ sufficiently large.

We have the following estimate for $N(d,r)$ due to Soundararajan.
\begin{thm}[{\cite[Theorem 1.1]{sound}}]\label{th:sound}
For any $\log_q{d \log^2 d} \leq r \leq d$, we have, uniformly in $q$,
\[
N(d,r) = q^d \rho \left(\frac{d}{r}\right) q^{ O \left( \frac{d \log d}{r^2} \right)}.
\]
\end{thm}
In particular, we have $N(d,r) \sim q^d \rho(d/r)$ as $d \rightarrow \infty$ and $\frac{r}{\sqrt{d \log d}} \rightarrow \infty$.

\section{Proof of Theorem \ref{th:main} and Corollary \ref{cor:main}}
Let us introduce auxiliary functions
\[
\calM(z,\chi)= \prod_{\deg(P) \leq r} \left( 1-\chi(P)z^{\deg(P)} \right) ^{-1}
\] 
and
\[
\calN(z,\chi)=\calL(z,\chi) / \calM(z,\chi).
\] 

Note that for $|z| <1/q$, we have
\[
\calM(z,\chi)= 1 + \sum_{k=1}^{\infty} \left( \sum_{f \in \calP(k,r)} \chi(f) \right) z^{k}. 
\]
Also, by the Euler product formula we have 
\[
\calN(z, \chi)= \prod_{\deg(P)>r} \left( 1-\chi(P) z^{\deg(P)} \right)^{-1}.  
\]
Let $0< R <1/q$ be arbitrary, and $C_R$ be the circle centered at 0 with radius $R$. By the Cauchy integral formula, we have
\begin{eqnarray}
 A(d,\chi) - \sum_{f \in \calP(d,r)} \chi(f) &=& \frac{1}{2 \pi i} \int_{C_R} \left( \calL(z,\chi) - \calM(z,\chi) \right) z^{-d-1} dz \nonumber \\
 &=& \frac{1}{2 \pi i} \int_{C_R} \calM(z,\chi) \left( \calN(z,\chi) - 1 \right) z^{-d-1} dz \label{eq:diff}
\end{eqnarray}
Therefore, it suffices to bound $\calM(z,\chi)$ and $\calN(z,\chi)-1$ on $C_R$.

By Lemma \ref{mertens}, we have $\calM(z,\chi) = O(r)$ on $C_R$.

On $C_R$, we have
\begin{eqnarray}
\left| \log N(z,\chi) \right| &=& \left| \sum_{\deg(P)>r} - \log \left( 1-\chi(P) z^{\deg(P)} \right) \right| \nonumber \\
&=& \left| \sum_{\deg(P)>r} \sum_{m=1}^{\infty} \frac{1}{m} \chi(P^m) z^{m\deg(P)} \right| \nonumber \\
&=& \left| \sum_{m=1}^{\infty} \sum_{k=r+1}^{\infty} \frac{z^{mk}}{m} \sum_{P \in I_k} \chi(P^m) \right|. \label{eq:2parts}  
\end{eqnarray}

We break the above sum into two parts, $m=1$ and $m\geq 2$. By Theorem \ref{th:hsu}, the contribution of $m=1$ in (\ref{eq:2parts}) is
\begin{eqnarray}
\left| \sum_{k=r+1}^{\infty} z^{k} \sum_{P \in I_k} \chi(P) \right| &\ll& n \sum_{k=r+1}^{\infty} \frac{q^{k/2}R^k}{k} \nonumber \\
 &\ll& \frac{n}{r} (Rq^{1/2})^r \label{eq:part1}
\end{eqnarray}
since $Rq^{1/2} < q^{-1/2} \leq 2^{-1/2}$.

The contribution of $m \geq 2$ in (\ref{eq:2parts}) is
\begin{eqnarray}
\left| \sum_{m=2}^{\infty} \sum_{k=r+1}^{\infty} \frac{z^{mk}}{m} \sum_{P \in I_k} \chi(P^m) \right| 
&\leq & \sum_{k=r+1}^{\infty} \sum_{m=2}^{\infty} \frac{1}{m q^{mk}} \cdot \frac{q^k}{k}   \nonumber \\
&\leq & \sum_{k=r+1}^{\infty} \sum_{m=2}^{\infty} \frac{1}{q^{(m-1)k}}  \nonumber \\
& \ll & \sum_{k=r+1}^{\infty} \frac{1}{q^k} \nonumber \\
& \ll & \frac{1}{q^r} \label{eq:part2}.
\end{eqnarray}

From (\ref{eq:part1}) and (\ref{eq:part2}), we have
\begin{equation}
\left| \log N(z,\chi) \right| \ll \frac{n}{r} (Rq^{1/2})^r+1/q^r \ll \frac{n}{r} (Rq^{1/2})^r.
\end{equation}

We have $\frac{n}{r} (Rq^{1/2})^r \leq \frac{n}{r}q^{-r/2} \leq 1$ if $r \geq 2 \log_q n$. Therefore, as long as $r \geq 2 \log_q n$, we have $|\calN(z,\chi)-1|=O\left(\frac{n}{r} (Rq^{1/2})^r\right)$ on $C_R$.
Combining this with (\ref{eq:diff}), we have
\begin{eqnarray} 
 A(d,\chi) - \sum_{f \in \calP(d,r)} \chi(f) &=& O \left( r \cdot \frac{n}{r} (Rq^{1/2})^r \cdot R^{-d} \right) \nonumber \\
 &=& O (n (R q^{1/2})^r R^{-d}) \label{eq:est}.
\end{eqnarray}
Now notice that all the above estimates are independent of $R<1/q$. Hence, letting $R$ tend to $1/q$, we obtain the bound
\[
 A(d,\chi) - \sum_{f \in \calP(d,r)} \chi(f) = O \left( nq^{-r/2} q^d \right)
\]
as desired.

\begin{remark}
 The estimate (\ref{eq:est}) remains valid in the wider range $r \geq 2 \log_q n - O(\log_q \log_q n)$, and Theorem \ref{th:main} could have been extended to this range.
 Unfortunately, when $r \leq 2 \log_q n$, the error term $O \left( nq^{-r/2} q^d \right)$ becomes larger than trivial. 
 The exponent of $q$ in the error term, as well as the coefficient of $\log_q n$ in Theorem \ref{th:main} are the best that can be achieved using our method, since it comes ultimately from Theorem \ref{th:hsu}.
\end{remark}
Corollary \ref{cor:main} follows immediately from Theorems \ref{th:main} and \ref{th:sound}, since clearly $2 \log_q n \geq \log_q (d \log^2 d)$.

\section{Proof of Theorem \ref{th:main2}}
Theorem~\ref{th:main2} is proved in a straightforward manner from Corollary \ref{cor:main}. 
The proof technique is credited to Burgess \cite{Bur62} and has been used in the context of irreducible polynomials in \cite{hsu}.

Recall that $N=q^n$. Define $A(N-1)=\left\{m:m|N-1, m\ \text{squarefree}\right\}$, so that $|A(N-1)|=2^{\omega(N-1)}$. 
Let $\calQ \subseteq A_d$ be the set of primitive elements of $\F$. 

Set $r=2\log_qn+2\log_q 1/\eps$ and $d=r\frac{C\log 1/\eps}{\log \log 1/\eps}$ for some $C$ to be specified later. 

First, we observe that the error term $\eps_{q,d,r,n}$ provided by Corollary \ref{cor:main} is $O_q(\eps)$. 

Indeed, $nq^{-r/2} = \eps$. By (\ref{eq:rho}), there is a constant $C_1$ such that for all $u \geq C_1$, $\rho(u) \leq e^{-u \log u}$. 
Let $C_2$ be a constant such that the first error term in Corollary \ref{cor:main} is bounded by $q^{C_2 \frac{d \log d}{r^2}}$.

We may assume that $\epsilon$ is sufficiently small, so that 
\[
 C_1 \leq \frac{C\log 1/\eps}{\log \log 1/\eps}.
\] 
Then, 
\[
\rho \left( \frac{d}{r} \right) q^{C_2 \frac{d\log d}{r^2}} \leq \exp \left( - \frac{d}{r} \log \frac{d}{r} + C_2 \frac{d\log d \log q}{r^2} \right).
\]
We claim that the left hand side is $\leq \epsilon$, which is true if 
\begin{equation}\label{eq:c1}
\frac{d}{r} \log \frac{d}{r} \geq 2\log \frac{1}{\eps} 
\end{equation} 
and 
\begin{equation}\label{eq:c2}
C_2 \frac{d\log d \log q}{r^2} \leq \frac{d}{2r} \log \frac{d}{r}.
\end{equation} 

Clearly, (\ref{eq:c1}) holds for some absolute constant $C>0$. Also, (\ref{eq:c2}) holds for $\eps$ sufficiently small (depending on $C_2$). Thus, 
\begin{equation}\label{eq:eps} \eps_{q,d,r,n}=O_q(\eps)\end{equation}
Let $\chi_0$ be the principal character on $\F_{q^n}$. For any $m|N-1$, we have $\sum_{\chi:\chi^m= \chi_0}\chi(x)=m$ if $x$ is an $m$-th power residue and zero otherwise. Using this, we build our indicator function for primitive elements in $\F_{q^n}$. 
Let $f(x)$ be the indicator function for primitive roots in $\F_{q^n}$. Then, it is easy to check that $$f(x)=\sum_{m|N-1}\frac{\mu(m)}{m}\sum_{\chi:\chi^m=1}\chi(x)$$ 
Now,
\begin{eqnarray*}
\sum_{x \in A_d}f(x)&=&\sum_{x \in A_d}\sum_{m|N-1}\frac{\mu(m)}{m}\sum_{\chi:\chi^m=\chi_0}\chi(x)\\
&=&\sum_{m|N-1}\frac{\mu(m)}{m}\sum_{\chi:\chi^m=\chi_0} \sum_{x \in A_d} \chi(x)\\
&=&\sum_{m|N-1}\frac{\mu(m)}{m}\left(q^d\right)+\sum_{m|N-1}\frac{\mu(m)}{m}\sum_{\chi \neq \chi_0:\chi^m=1}\sum_{x \in A_d}\chi(x)\\
& = &q^d\frac{\phi(N-1)}{N-1} + O \left( 2^{\omega(N-1)} \eps_{q,d,r,n} q^d \right)\ \ \text{(by Corollary \ref{cor:main})}
\end{eqnarray*}
Hence, by (\ref{eq:eps}), $$\l|\frac{|\calQ(d)|}{q^d}-\frac{\phi(N-1)}{N-1}\r| = 2^{\omega(N-1)}O_q(\eps).$$





\section{Proof of Theorem \ref{thm:iwimprove}}
In this section, we shall see how to improve the probability of outputting a primitive element from the set of monic degree $d$ polynomials. We shall use a generalization of Iwaniec's shifted sieve \cite{iw}, proved by Shoup \cite{Shoup}. We state it next.

\begin{thm}[Prop. 1 in \cite{Shoup}]\label{thm:iw}Let $\Gamma$ be a finite set. 
Let $U:\Gamma \rightarrow \Z$, $W:\Gamma \rightarrow \R_{\geq 0}$. 
Let $p_1,\ldots p_{l}$ be distinct primes with $\Pi=\prod_{i=1}^{l}p_i$. 
Define $$T=\sum_{\gamma \in \Gamma:\  \gcd(U(\gamma),\Pi)=1}W(\gamma)$$ and for $m|\Pi$, 
$$S_m=\sum_{\gamma \in \Gamma:\  U(\gamma)=0 \pmod m}W(\gamma).$$ 
Suppose there are $A,B$ such that for all $m|\Pi$, $$\l|S_m-A/m\r| \leq B.$$ Then 
$$T \geq c_1A/(\log l+1)^2-c_2l^2B,$$ where $c_1, c_2$ are absolute positive constants.
\end{thm}

%

\begin{proof}[Proof of Theorem~\ref{thm:iwimprove}]
Let $g$ be an arbitrary primitive element of $\F_{q^n}^\times$. Let $\Gamma=A_d$ and for every $f \in A_d$, let $U(f)$ denote the discrete logarithm of $f$ with base $g$. Let $W(f)=1$. 
Let $N-1=\prod_{i=1}^{l} p_i^{\alpha_i}$ be the prime factorization of $N-1$. Set $\Pi=\prod_{i=1}^{l}p_i$. Let $m|\Pi$ and $\chi$ be a multiplicative character of $\F_{q^n}^\times$ of order $m$. We have
\begin{eqnarray*}
S_m&=& \sum_{f \in A_d:  U(f)=0 \pmod m}1\\
&=&\sum_{f \in A_d} \frac{1}{s} \sum_{i=0}^{m-1}\chi^i(f)\\
&=&\frac{1}{m} \sum_{i=0}^{m-1} \sum_{f \in A_d} \chi^i(f) \\
\end{eqnarray*}
By Corollary~\ref{cor:main}, we have $|S_m - \frac{q^d}{m} | \leq q^d \eps$, where $\eps=\eps_{q,d,r,n}$.
Thus, with $A=q^d$ and $B=q^d\eps$, by Theorem~\ref{thm:iw}, we have $$T \geq c_1q^d/(\log l+1)^2-c_2l^2q^d\eps.$$
Also, note that by the way $T$ is defined, $T$ is also the cardinality of the primitive elements in $A_d$. Thus, we have $$\frac{T}{q^d} \geq c/(\log l+1)^2$$ as long as $\eps<\frac{c}{l^2 (\log l + 1)^2}$,
for a suitably chosen $c>0$.
\end{proof}

\section{Acknowledgement}
We thank Ariel Gabizon, Felipe Voloch and David Zuckerman for helpful discussions, and Kannan Soundararajan for kindly sending us the unpublished manuscript \cite{sound}. The first author would also like to thank his advisor, David Zuckerman for his constant support and encouragement throughout the project.

\end{document}